\documentclass[12pt]{amsart}

\usepackage{enumerate, amsmath, amsthm, amsfonts, amssymb, xy,  mathrsfs, graphicx, paralist, fancyvrb, eucal}
\usepackage[usenames, dvipsnames]{xcolor}
\usepackage[margin=1in]{geometry} 
\usepackage[bookmarks, colorlinks=true, linkcolor=blue, citecolor=blue, urlcolor=blue]{hyperref}

\input xy
\xyoption{all}

\numberwithin{equation}{section}
\newtheorem{theorem}[equation]{Theorem}

\newtheorem{proposition}[equation]{Proposition}
\newtheorem{lemma}[equation]{Lemma}
\newtheorem{corollary}[equation]{Corollary}

\theoremstyle{definition}
\newtheorem{rmk}[equation]{Remark}
\newenvironment{remark}[1][]{\begin{rmk}[#1] \pushQED{\qed}}{\popQED \end{rmk}}
\newtheorem{eg}[equation]{Example}

\newtheorem{defn}[equation]{Definition}

\newenvironment{subeqns}[1][]{\addtocounter{equation}{-1}
\begin{subequations}

}{\end{subequations}}

\newcommand{\bA}{\mathbf{A}}

\newcommand{\bC}{\mathbf{C}}
\newcommand{\cC}{\mathcal{C}}

\newcommand{\sC}{\mathscr{C}}

\newcommand{\cE}{\mathcal{E}}

\newcommand{\rE}{\mathrm{E}}

\newcommand{\bF}{\mathbf{F}}

\newcommand{\bG}{\mathbf{G}}

\newcommand{\rH}{\mathrm{H}}

\newcommand{\rJ}{\mathrm{J}}

\newcommand{\bK}{\mathbf{K}}

\newcommand{\cL}{\mathcal{L}}

\newcommand{\cM}{\mathcal{M}}

\newcommand{\sM}{\mathscr{M}}

\newcommand{\cO}{\mathcal{O}}

\newcommand{\bP}{\mathbf{P}}
\newcommand{\cP}{\mathcal{P}}

\newcommand{\bQ}{\mathbf{Q}}

\newcommand{\rR}{\mathrm{R}}

\newcommand{\cV}{\mathcal{V}}

\newcommand{\cX}{\mathcal{X}}

\newcommand{\bZ}{\mathbf{Z}}

\newcommand{\fe}{\mathfrak{e}}
\newcommand{\re}{\mathrm{e}}

\newcommand{\fh}{\mathfrak{h}}
\newcommand{\rh}{\mathrm{h}}

\newcommand{\bk}{\mathbf{k}}

\newcommand{\fm}{\mathfrak{m}}

\renewcommand{\phi}{\varphi}
\renewcommand{\emptyset}{\varnothing}
\newcommand{\eps}{\varepsilon}

\renewcommand{\tilde}[1]{\widetilde{#1}}
\newcommand{\ol}[1]{\overline{#1}}

\newcommand{\arxiv}[1]{\href{http://arxiv.org/abs/#1}{{\tt arXiv:#1}}}

\makeatletter
\def\Ddots{\mathinner{\mkern1mu\raise\p@
\vbox{\kern7\p@\hbox{.}}\mkern2mu
\raise4\p@\hbox{.}\mkern2mu\raise7\p@\hbox{.}\mkern1mu}}
\makeatother

\DeclareMathOperator{\im}{image}

\renewcommand{\hom}{\operatorname{Hom}}

\DeclareMathOperator{\rank}{rank}

\DeclareMathOperator{\Aut}{Aut}

\DeclareMathOperator{\ad}{ad}
\DeclareMathOperator{\Spec}{Spec}

\DeclareMathOperator{\Pic}{Pic}

\newcommand{\GL}{\mathbf{GL}}
\newcommand{\SL}{\mathbf{SL}}
\newcommand{\Sp}{\mathbf{Sp}}

\newcommand{\Gr}{\mathbf{Gr}}

\newcommand{\fsl}{\mathfrak{sl}}

\newcommand{\Flag}{\mathrm{Flag}}

\DeclareMathOperator{\Sel}{Sel}
\newcommand{\ks}{\bk^{\rm sep}}

\newcommand{\PGL}{\mathbf{PGL}}

\newcommand{\SU}{\mathrm{SU}}
\newcommand{\bmu}{\boldsymbol{\mu}}

\newcommand{\df}[1]{{\bf \textsf{#1}}}

\newcommand{\st}{\mathrm{st}}

\usepackage{comment}

\title[$\bigwedge^3(9)$ and genus $2$ curves]{Invariant theory of $\bigwedge^3(9)$ and genus $2$ curves}
\date{July 24, 2017}

\author{Eric M. Rains}
\address{Department of Mathematics, California Institute of Technology}
\email{rains@caltech.edu}

\author{Steven V Sam}
\address{Department of Mathematics, University of Wisconsin, Madison}
\email{svs@math.wisc.edu}

\thanks{SS was partially supported by a Miller research fellowship and NSF DMS-1500069.}


\begin{document}

\maketitle

\begin{abstract}
Previous work established a connection between the geometric invariant theory of the third exterior power of a 9-dimensional complex vector space and the moduli space of genus 2 curves with some additional data. We generalize this connection to arbitrary fields, and describe the arithmetic data needed to get a bijection between both sides of this story.
\end{abstract}

\section{Introduction}

This paper is a companion to our previous paper \cite{coble}. We begin by briefly recalling what was done there. Given a genus $2$ curve $C$ over a field $\bk$, let $\SU_3(C)$ be the coarse moduli space of rank $3$ semistable vector bundles on $C$. It admits a degree $2$ map $\SU_3(C) \to \bP^8$ which is branched along a sextic hypersurface. Remarkably, the singular locus of the projective dual of this sextic is a surface which is isomorphic to the Jacobian of $C$ over the algebraic closure of $\bk$. This story has been developed over algebraically closed fields of characteristic $0$ in \cite{ortega, minh} and connected to the invariant theory of the action of $\SL_9(\bk)$ on $\bigwedge^3 \bk^9$ in \cite{GS, GSW}. In \cite{coble}, the setting is generalized to arbitrary fields, and the purpose of this paper is to extend the invariant-theoretic aspects.

More precisely, let $V$ be a $9$-dimensional vector space over an arbitrary field $\bk$ and consider the action of $\SL(V)$ on $\bigwedge^3 V$. Given a stable (in the sense of geometric invariant theory) element $\gamma \in \bigwedge^3 V$, we generalize the constructions in \cite{GSW, GS} to produce:
\begin{itemize}
\item a genus $2$ curve $C$ with a Weierstrass point $P \in C(\bk)$ and
\item a cubic hypersurface in $\bP(V^*)$ whose singular locus is a smooth surface $X$, 
\end{itemize}
such that $X$ is isomorphic to the Jacobian $\rJ(C)$ of $C$ over the algebraic closure of $\bk$. In fact, we also get some interesting arithmetic data:
\begin{itemize}
\item a $3$-covering $X \to \rJ(C)$ which becomes the multiplication by $3$ map over $\ol{\bk}$, i.e., an element in $\rH^1(\bk; \rJ(C)[3])$; furthermore, it lies in the kernel of a map $\rH^1(\bk; \rJ(C)[3]) \to \rH^1(\bk; \SL_9/\bmu_3)$.
\end{itemize}
Conversely, given this data, we show how to construct a stable element in $\bigwedge^3 V$ (which is only well-defined up to scalar multiple and the action of $\SL(V)$). A bulk of the work in this paper is to show that these two constructions are inverse to one another.

Our work is partially motivated by recent work in ``arithmetic invariant theory'' (see \cite{AIT} for example). One goal is to count arithmetic objects of interest, and the first step in many of these cases is to parametrize them by orbits in a linear space. This first step is achieved here; when $\bk$ is a global field, we show that the $3$-Selmer group of $\rJ(C)$ is a subgroup of the kernel of $\rH^1(\bk; \rJ(C)[3]) \to \rH^1(\bk; \SL_9/\bmu_3)$, so that, in fact, they are parametrized by special kinds of orbits in $\bigwedge^3 V$. Following the analogies of previous work in the area, we may hope to count the average size of this $3$-Selmer group using $\bigwedge^3 V$.

Here is a brief overview of the contents. In \S\ref{sec:invt-theory}, we work out the aspects of the invariant theory of $\bigwedge^3 V$ which are needed in the rest of the paper. In \S\ref{sec:3covers}, we generalize the construction of \cite{GS, GSW} to arbitrary fields, i.e., produce the data above starting from a stable element $\gamma$. In \S\ref{sec:trivector}, we provide a construction in the reverse direction: starting from the data above, we produce a stable element $\gamma$, and in \S\ref{sec:bijection}, we show that these two constructions are inverse to one another. Finally, in \S\ref{sec:complements}, we discuss a few additional topics: Selmer groups, ordinary curves, and an explicit model for the $3$-torsion of $\rJ(C)$ given $\gamma$ above.

\subsection{Notation}

$\bk$ is a field and $R$ is a complete discrete valuation ring (DVR from now on) of characteristic $0$ whose residue field is $\bk$; the quotient field of $R$ is denoted $\bK$. Write $\ks$ for a separable closure of $\bk$.

If $G$ is a group scheme defined over $\bk$, we let $\rH^*(\bk; G)$ denote the flat cohomology of $G$. When $G$ is smooth, this coincides with the Galois cohomology of $G$, but we will not have any use for Galois cohomology of non-smooth group schemes.

\section{Invariant theory preliminaries} \label{sec:invt-theory}

\subsection{Geometric invariant theory review}

Let $G$ be a reductive group acting linearly on a vector space $V$. A point $u \in V$ is \df{stable} if its stabilizer subgroup in $G$ is finite and its orbit is closed, and it is \df{semistable} if $0$ is not in the closure of its orbit. If $u$ is not semistable, then it is \df{unstable}. Hence, an element is ``non-stable'' if it is unstable or if it is semistable, but not stable.

The Hilbert--Mumford criterion says that $u$ is stable if and only if $\lim_{t \to 0} \rho(t).u$ does not exist for any $1$-parameter subgroup $\rho \colon \bG_m \to G$, and that $u$ is semistable if and only if $\lim_{t \to 0} \rho(t).u$ does not exist, or it is nonzero whenever the limit exists. Note that we will work over arbitrary fields, but one must consider all $1$-parameter subgroups which are defined over the algebraic closure.

The set of unstable points form a Zariski closed set, and is the zero locus of all positive degree $G$-invariant homogeneous polynomials on $V$. Similarly, the set of non-stable points form a Zariski closed set. Finally, two points $x,y \in V$ are \df{S-equivalent} if $f(x) = f(y)$ for all homogeneous $G$-invariant polynomials $f$ on $V$, and they are \df{projectively S-equivalent} if $\alpha x$ is S-equivalent to $y$ for some $\alpha \ne 0$. If $x$ and $y$ are S-equivalent semistable points, then their orbit closures have a semistable point in common. Furthermore, the orbit closure of any semistable point $x$ contains a unique closed orbit of semistable points, and if $x$ is not stable, then neither are the points of this closed orbit.

Let $V_9$ denote a vector space of dimension $9$ with basis $e_1, \dots, e_9$. The group $G = \SL(V_9)/\bmu_3$ acts on $\bigwedge^3(V_9)$, and the invariant theory of this representation is the main focus of this paper (see also \cite{EV} for earlier work). It has a natural basis of monomials $e_i \wedge e_j \wedge e_k$ (with $1 \le i < j < k \le 9$), and we will use $[ijk]$ as shorthand for this monomial.

For the following, see \cite{GGR}.

\begin{proposition} \label{prop:S-equiv}
Over an algebraically closed field, every element of $\bigwedge^3(V_9)$ is
S-equivalent to an element $\gamma_c$ of the form
\begin{align*}
[267]+[258]+[348]+[169]+[357]+[249]+[178]+[456]\\
-c_3 [257]
-c_6 [247]
+c_9 [148]
-c_{12} [147]
+c_{15} [235]
+c_{18} [145]
+c_{24} [134]
+c_{30} [123],
\end{align*}
where if $2$ is invertible we may take $c_3=c_9=c_{15}=0$ and if $5$ is
invertible we may take $c_6=0$.  Two such elements are projectively
S-equivalent if and only if the corresponding pairs $(C_c,P_c)$ are isomorphic, where $C_c$ is the curve
\[
C_c: x^2+z^5 + c_3 x z^2+ c_6 z^4
+ c_9 x z + c_{12} z^3+ c_{15} x + c_{18} z^2  + c_{24} z + c_{30} = 0,
\]
and $P_c$ is the point at infinity.
\end{proposition}

\begin{remark}
Above, we see that projective S-equivalence classes classify pairs $(C,P)$ where $C$ is a genus $2$ curve and $P \in C(\bk)$ is a rational Weierstrass point. In fact, one can show that the S-equivalence classes themselves classify triples $(C,P,\phi)$ where $\phi \colon \omega_C\otimes \cO_P\cong \cO_P$ specifies a nonzero tangent vector at $P$. We omit the details, as we have not been able to figure out how to build $\phi$ into the construction below, and can thus only work at the level of projective S-equivalence. 
\end{remark}

\begin{remark}
The only way in which the results below will logically depend on Proposition~\ref{prop:S-equiv} is the claim that elements corresponding to isomorphic $(C_c,P_c)$ pairs are projectively equivalent.  This is quite easy to check computationally: any isomorphism of pairs has the form
\[
(x,z)\mapsto (\alpha^5 x + b_3 z^2 + b_9 z + b_{15},\alpha^2 z + b_6),
\]
and it is easy to find an equivalence between the corresponding trivectors given the ansatz that the element of $\GL_9$ be upper-triangular.  As for the other claims of Proposition~\ref{prop:S-equiv}, not only will they not be used below, but in fact for stable $\gamma$, they are easy consequences of our results!  (Of course, this is more in the nature of a proof, while \cite{GGR} gives an actual derivation.)
\end{remark}

\begin{proposition}\label{prop:homogeneity}
If $F$ be an $\SL_9$-invariant section of $\cO_{\bP(\bigwedge^3(V_9))}(d)$. Then $F(\gamma_c)$ is a homogeneous polynomial of degree $d$ in $c_3, \dots, c_{30}$, with $\deg(c_i)=i$.
\end{proposition}

\begin{proof}
The $1$-parameter subgroup of $\GL_9$ of weight $(15, 9, 6 , 3 ,0,-3, -6,-9,-12)$ preserves the space of elements $\gamma_c$ and acts on each $c_i$ by $t^i$.  Since $F$ is an $\SL_9$-invariant of degree $d$, $F(g\gamma) = \det(g)^{d/3} F(\gamma)$ for any $g\in \GL_9$, and thus the $1$-parameter subgroup multiplies $F$ by $t^d$, so that $F(\gamma_{t\cdot c}) = t^d F(\gamma_c)$, implying the desired homogeneity.
\end{proof}

\subsection{Cartan subspaces} \label{ss:cartan}

Assume the characteristic of $\bk$ is different from $3$.
Let $G = \SL(V_9) / \bmu_3$ and let $\fe_8$ be the split Lie algebra of type $\rE_8$ and let $\Gamma$ be its simply-connected group. We have a $\bZ/3$-graded decomposition
\begin{align} \label{eqn:e8-decomp}
\fe_8 = \fsl(V_9) \oplus \bigwedge^3 V_9 \oplus \bigwedge^6 V_9.
\end{align}
The decomposition \eqref{eqn:e8-decomp} corresponds to an order $3$ automorphism $\theta$ of $\Gamma$ such that $G = \Gamma^\theta$ and $\bigwedge^3 V_9$ is one of the nontrivial eigenspaces of $\theta$ acting on $\fe_8$. More explicitly, pick a set of simple roots $\alpha_1,\dots,\alpha_8$ for the root system of $\fe_8$. Then the height of a root is the sum of its coefficients when expressed as a sum of the $\alpha_i$, and the $\bZ/3$-grading comes from taking the height modulo $3$.

The $4$ dimensional subspace $\fh$ of $\bigwedge^3 V_9$ spanned by 
\begin{equation} \label{eqn:cartan}
\begin{split}
[123]+[456]+[789], \qquad [147]+[258]+[369],\\ 
[159]+[267]+[348], \qquad [168]+[249]+[357],
\end{split}
\end{equation}
is the \df{standard Cartan subspace}. It may be helpful to visualize this in terms of the finite geometry $\bP_{\bF_3}^2$, namely, each basis vector is a sum over all lines in a direction of the following table:
\[
\begin{array}{ccc}
1&2&3\\
4&5&6\\
7&8&9
\end{array}
\]
This carries the action of the Weyl group $W = N(\fh) / Z(\fh)$ (normalizer modulo centralizer).

\begin{proposition} \label{prop:chevalley}
If $\bk$ has characteristic $0$, the restriction map
\[
\bk[\bigwedge^3 V_9]^G \xrightarrow{\cong} \bk[\fh]^W
\]
is an isomorphism, and both are polynomial rings generated by elements of degrees $12, 18, 24, 30$.
\end{proposition}

\begin{proof}
See \cite[Theorem 7]{vinberg} for the isomorphism, and see \cite[\S 9]{vinberg} for the degrees of the invariants.
\end{proof}

When $\bk = \bC$, the quotient space $\fh/W$ is classically known to parametrize genus $2$ curves together with a choice of Weierstrass point, see \cite[\S 4]{dolgachev-lehavi}.

$W$ is a complex reflection group (the reflections have order 3), and there are 40 reflection hyperplanes. With respect to the $4$ basis vectors in \eqref{eqn:cartan} for the standard Cartan subspace, the matrix representation of the reflection group in characteristic $0$ is given in \cite[\S 3.1]{GS}. Each reflection hyperplane is in the orbit of the hyperplane spanned by the first $3$ basis vectors (see \cite[Table 1]{GS}). As an abstract finite group, we have an isomorphism $W \cong \bZ/3 \times \Sp_4(\bF_3)$.

\begin{lemma} \label{lem:closedorbit}
Suppose $\bk$ has characteristic $0$. If $x$ is semistable and $Gx$ is closed, then $Gx \cap \fh \ne 0$.
\end{lemma}

\begin{proof}
Combine \cite[Proposition 4]{vinberg} and \cite[Corollary of Theorem 1]{vinberg}.
\end{proof}

\begin{proposition} \label{prop:refl-posdim}
Any element of a reflection hyperplane in the standard Cartan subspace has a positive-dimensional stabilizer subgroup in $G$.
\end{proposition}

\begin{proof}
In positive characteristic, lift our element over the DVR $R$ to characteristic $0$ and use semicontinuity of stabilizer dimension to reduce the proof to the case of characteristic $0$.

The reflection hyperplanes form a single orbit under the reflection group, so it suffices to consider a single one. From the discussion above, we may assume that this hyperplane is the span of $[123]+[456]+[789], [147]+[258]+[369], [159]+[267]+[348]$. Then for any $t$, the diagonal matrix with entries $(t^{-2}, t, t, t, t, t^{-2}, t, t^{-2}, t)$ stabilizes each of these $3$ basis vectors, and hence any element in this hyperplane. So the stabilizer of any element has positive dimension.
\end{proof}

\begin{proposition} \label{prop:stable-refl}
An element $u$ in the standard Cartan subspace is stable if and only if it does not lie in any reflection hyperplanes.
\end{proposition}

\begin{proof}
The standard Cartan subspace is the intersection of a Cartan subalgebra of $\fe_8$ with $\bigwedge^3 V_9$ and none of the reflection hyperplanes of the Cartan subalgebra of $\fe_8$ contain the standard Cartan subspace (this follows from the discussion in \cite[\S 3]{elkies}), so $u$ is contained in the complement of reflection hyperplanes in a Cartan subalgebra of $\fe_8$, which means that it is stable under the action of $\Gamma$. The Hilbert--Mumford criterion implies that $u$ is stable as an element of $\bigwedge^3 V_9$ under the action of $G$. Conversely, we have already seen that any element in a reflection hyperplane has a positive-dimensional stabilizer, so cannot be stable.
\end{proof}

\subsection{Stable elements}

\begin{lemma} \label{lem:deg120}
In characteristic $0$, the locus of non-stable elements of $\bigwedge^3 V_9$ is contained in an irreducible $G$-invariant hypersurface of degree $120$.
\end{lemma}

\begin{proof}
Let $x$ be a semistable, but not stable point, and let $y$ be a point in
its orbit closure such that $Gy$ is closed. Then $Gy \cap \fh \ne 0$ by
Lemma~\ref{lem:closedorbit}, and we may assume $y \in \fh$. By
Proposition~\ref{prop:stable-refl}, $y$ lies on a reflection
hyperplane. Let $f$ be the product of the linear forms vanishing on the
reflection hyperplanes of $\fh$, so $\deg f = 40$. The reflections
transform $f$ by a cube root of unity, so $f^3$ is the lowest degree
$W$-invariant vanishing on each reflection hyperplane. Let $\delta$ be the
$G$-invariant function on $\bigwedge^3 V_9$ which corresponds to $f^3$
under the isomorphism in Proposition~\ref{prop:chevalley}. Then $\delta$
vanishes on $y$ since it restricts to $f^3$, and hence $\delta$ also
vanishes on $x$. Finally, $\delta$ is irreducible: if not, then each
component is cut out by a $G$-invariant since $G$ is connected, and would
restrict to a $W$-invariant function of degree $<120$ vanishing on some of
the reflection hyperplanes, but no such function exists.
\end{proof}

\begin{proposition} \label{prop:non-stable-hyper}
\begin{enumerate}[\rm (a)]
\item An element $u \in \bigwedge^3 V_9$ is non-stable if and only if there exists a $6$-dimensional subspace $U \subset V_9$ such that $\gamma \in \bigwedge^3 U + \bigwedge^2 U \otimes (V_9/U)$.

\item The set of non-stable elements in $\bigwedge^3 V_9$ is an irreducible hypersurface which is set-theoretically defined by a polynomial of degree $120$. This hypersurface is reduced in characteristic $0$.
\end{enumerate}
\end{proposition}

\begin{proof}
Let $Z$ be the set of $u$ such that there exists a 6-dimensional subspace $U \subset V_9$ such that $\gamma \in \bigwedge^3 U + \bigwedge^2 U \otimes (V_9/U)$.

Pick $\gamma \in Z$ with $U$ as above. Pick a basis $u_1, \dots, u_6$ for $U$ and extend it to a basis $u_1, \dots, u_9$ for $V_9$. Then $\gamma$ is a sum of trivectors $[ijk]$ where $|\{i,j,k\} \cap \{7,8,9\}| \le 1$. In particular, given the diagonal 1-parameter subgroup $\rho(t) = (t^3, t^3, t^3, t^3, t^3, t^3, t^{-6}, t^{-6}, t^{-6})$, we have $\lim_{t \to 0} \rho(t) \cdot \gamma$ exists, and is the result of throwing away the $[ijk]$ where $|\{i,j,k\} \cap \{7,8,9\}| = 0$. By the Hilbert--Mumford criterion, $\gamma$ is non-stable, so $Z$ is contained in the non-stable locus.

Let $P$ be the stabilizer in $\GL(V_9)$ of the subspace $e_1, \dots, e_6$ and let $E$ be the span of $e_1, \dots, e_6$. Then the span of $\bigwedge^3 E$ and $\bigwedge^2 E \otimes (V_9/E)$ is a $P$-submodule of $\bigwedge^3 V_9$ and by algebraic induction, this $P$-submodules becomes a rank $65$ vector bundle $\cE$ which is a subbundle of $\bigwedge^3 V_9 \times \Gr(6,V_9)$ where $\Gr(6,V_9)$ is the Grassmannian of $6$-dimensional subspaces of $V_9$. By the discussion above, the image of the projection $\pi \colon \cE \to \bigwedge^3 V_9$ is $Z$. In particular, $Z$ is irreducible. Let $\xi \subset \bigwedge^3 V_9^* \times \Gr(6,V_9)$ be the annihilator of $\cE$; then the Koszul complex $\bigwedge^\bullet \xi$ is a locally free resolution of $\cE$ as a subscheme of $\bigwedge^3 V_9 \times \Gr(6,V_9)$, and so its derived pushforward with respect to $\pi$ has the same Euler characteristic as $\rR \pi_* \cO_{\cE}$ (for a discussion of this, see \cite[Chapter 5]{weyman}). More specifically, everything respects the natural $\bZ$-grading, so we can calculate the Hilbert series of $\rR \pi_* \cO_{\cE}$ as:
\[
\sum_{i \ge 0} (-1)^i \rH_{\rR^i \pi_* \cO_\cE}(t) = \sum_{i = 0}^{19} (-1)^i \chi(\Gr(6,V_9); \bigwedge^i \xi) \frac{t^i}{(1-t)^{84}}.
\]
The right hand side can be computed using Borel--Weil--Bott \cite[Corollary 4.1.7]{weyman} and yields
\[
\frac{1 + t^6 + t^9 + 81t^{18} - 84t^{19}}{(1-t)^{84}} = \frac{h(t)}{(1-t)^{83}}
\]
where
\begin{align*}
h(t) &= 84t^{18}+3t^{17}+3t^{16}+3t^{15}+3t^{14}+3t^{13}+3t^{12}+3t^{11}\\
& \qquad +3t^{10}+3t^9+2t^8+2t^7+2t^6+t^5+t^4+t^3+t^2+t+1.
\end{align*}
In particular, the support of $\rR \pi_* \cO_\cE$ has dimension $83$, and this support is $Z$. This matches the dimension of the total space of $\cE$, so generically, the map $\pi$ has $0$-dimensional fibers. Since $\pi$ is projective, this implies that the support of $\rR^i \pi_* \cO_\cE$ for each $i>0$ has dimension $\le 82$. In particular, the multiplicity of $\pi_* \cO_\cE$ is $h(1) = 120$, and the degree of $Z$ divides $120$. 

In characteristic $0$, we know that $Z$ is contained in an irreducible hypersurface of degree $120$ by Lemma~\ref{lem:deg120}, so we conclude that $Z$ coincides with this hypersurface. This proves (a) and (b) in characteristic $0$.

Now we prove (a) in general. What remains is to show that every non-stable element belongs to $Z$. Let $\gamma$ be a non-stable element. Let $R$ be a complete DVR with residue field $\bk$ and fraction field $\bK$ of characteristic $0$. Let $\rho$ be a 1-parameter subgroup of $G(\bk)$ such that $\lim_{t \to 0} \rho(t) \cdot \gamma$ exists. By changing basis, we may assume that the image of $\rho$ is contained in the diagonal matrices, and hence $\rho$ can be lifted to a 1-parameter subgroup $\tilde{\rho}$ of $G(R)$. The action of $\tilde{\rho}$ on $\bigwedge^3 R^9$ decomposes it into weight spaces which are free $R$-submodules, we are interested in the negative versus non-negative subspaces. The non-negative subspace corresponds to all elements which have a limit under the action of $\tilde{\rho}(t)$ for $t \to 0$ and its reduction to $\bk$ is the non-negative subspace of the action of $\rho$ on $\bigwedge^3 V_9$. So we can lift $\gamma$ to a non-stable element $\tilde{\gamma} \in \bigwedge^3 R^9$ such that $\tilde{\gamma}_\bK \in \bigwedge^3 \bK^9$ is also non-stable. By what we just showed, there exists a $6$-dimensional subspace $U \subset \bK^9$ such that $\tilde{\gamma}_\bK \in \bigwedge^3 U + \bigwedge^2 U \otimes (\bK^9/U)$. Since the Grassmannian is proper, $U$ can be lifted to a rank $6$ $R$-submodule $\tilde{U} \subset R^9$ such that $R^9/\tilde{U}$ is free. In particular, $\tilde{\gamma} \in \bigwedge^3 \tilde{U} + \bigwedge^2 \tilde{U} \otimes R^9/\tilde{U}$ since this is a closed condition on the fibers of $R$ and it is true generically. In particular, the special fiber of $\tilde{U}$ gives a subspace which shows that $\gamma \in Z$. 

By what was shown already, we know that $Z$ is an irreducible hypersurface whose degree divides 120, so we conclude that the same is true for the non-stable locus.
\end{proof}

\begin{proposition} \label{prop:stable-smooth}
$\gamma_c$ is stable if and only if $C_c$ is smooth.
\end{proposition}

\begin{proof}
If the curve $C_c$ is singular, then translating the singular point to $(0,0)$ gives a curve 
\[
C_{c'}: x^2+z^5 + c'_3 x z^2+ c'_6 z^4 + c'_9 x z + c'_{12} z^3+ c'_{15} x + c'_{18} z^2  + c'_{24} z + c'_{30} = 0.
\]
In particular, $c'_{30} = 0$ (since $(0,0)$ is a point) and $c'_{15}=c'_{24}=0$ (since the partial derivatives of $x$ and $z$ vanish at $(0,0)$). If we take $U = \langle e_4, e_5, e_6, e_7, e_8, e_9\rangle$, then $\gamma_{c'} \in \bigwedge^3 U + \bigwedge^2 U \otimes (V_9/U)$, so $\gamma_{c'}$ is non-stable by Proposition~\ref{prop:non-stable-hyper}, so the same is true for $\gamma_c$ since they are projectively S-equivalent by Proposition~\ref{prop:S-equiv}.

Consider the set of all pairs $(\gamma_c,U)$ with $\gamma_c\in \bigwedge^3
U + \bigwedge^2 U\otimes (V_9/U)$.  This is a closed subscheme of
$\bA^8\times \Gr(6,9)$, and is thus proper over $\bA^8$.  We claim that in
any characteristic, the total space is smooth of dimension $7$ and
irreducible.

The 1-parameter subgroup of Proposition \ref{prop:homogeneity} acts on this
scheme, and since the limit $t\to 0$ always exists in $\bA^8$, properness
implies that it exists in the scheme of pairs.  We claim that the limit
must, in fact, be $(\gamma_0,\langle e_4,\dots,e_9\rangle)$.  Indeed,
since we are taking a limit along a diagonal 1-parameter subgroup with
distinct eigenvalues, the limiting subspace is a coordinate subspace, and
there is only coordinate subspace that destabilizes $\gamma_0$.  Since the
limit point is independent of the starting point, we can bound the
dimension of every tangent space by computing its dimension at the limit.
This is straightforward linear algebra, and we find that it is indeed
7-dimensional.  Since we already know a 7-dimensional component and every
component meets the limit point, there can be no other components, and the
component we know is smooth.  In particular, the image of this scheme in
$\bA^8$ must be precisely the locus where $C_c$ is singular, as required.
\end{proof}

\section{Parametrizing $3$-coverings of abelian surfaces} \label{sec:3covers}

Let $(\bigwedge^3 V_9)_\st$ be the set of stable elements of $\bigwedge^3 V_9$ with respect to the $\SL(V_9)/\bmu_3$-action.

Fix $u \in (\bigwedge^3 V_9)_\st$. From this data, we will construct:
\begin{itemize}
\item a genus $2$ curve $C$ with a marked Weierstrass point $P \in C(\bk)$, 
\item a $3$-covering $\psi \colon X \to J$ (where $J = \rJ(C)$ is the Jacobian of $C$) such that $[\psi] \in \ker(\rH^1(\bk; \rJ(C)[3]) \to \rH^1(\bk; \SL(V_9)/\bmu_3))$.
\end{itemize}
Recall that $\psi \colon X \to J$ is a $3$-covering if $X$ is a torsor for $J$ and $\psi$ can be identified with the multiplication-by-$3$ map over an algebraic closure of $\bk$; $3$-coverings are classified by cohomology classes in $\rH^1(\bk; J[3])$ \cite[Proposition 3.3.2]{skorobogatov}.

To simplify notation, we will not label the objects by $u$, but we emphasize that all constructions depend on the $\PGL(V_9)$-orbit of $[u] \in \bP((\bigwedge^3 V_9)_\st)$.

Let $\bP(V_9^*)$ denote the space of lines in $V_9^*$. Then $V_9$ is the space of linear functions on $\bP(V_9^*)$, so we can treat $e_1, \dots, e_9$ as coordinate functions. Following \cite[\S 3.2]{GS}, we interpret $u \in \bigwedge^3 V_9$ as a family of $9 \times 9$ skew-symmetric matrices
\[
\Phi \colon V_9^* \to V_9 \otimes \cO_{\bP(V_9^*)}(1)
\]
over $\bP(V_9^*)$. In more details, given $u \in \bigwedge^3 V_9$, apply the comultiplication map $\bigwedge^3 V_9 \to \bigwedge^2 V_9 \otimes V_9$, use the natural surjection $V_9 \otimes \cO_{\bP(V_9^*)} \to \cO_{\bP(V_9^*)}(1)$, and interpret $\bigwedge^2 V_9$ as the space of skew-symmetric matrices $V_9^* \to V_9$. In particular, this construction is $\GL(V_9)$-equivariant, so acting by $\GL(V_9)$ amounts to a projective linear change of coordinates in $\bP(V_9^*)$.

Let $Y \subset \bP(V_9^*)$ be the locus where $\rank \Phi \le 6$. Let $X \subset \bP(V_9^*)$ be the locus where $\rank \Phi \le 4$. 

\begin{lemma}
$X$ is smooth of dimension $2$, and the locus where $\rank \Phi \le 2$ is empty.
\end{lemma}

\begin{proof}
If there is a point in the rank $2$ locus, then we can choose a basis so that it is the point $[1:0:0: 0:0:0: 0:0:0]$. So $u = [123] + u'$ where no monomial in $u'$ contains $e_1$. Then $(e_1\mapsto e_1+ \alpha e_2+ \beta e_3)$ is a $2$-dimensional subgroup of the stabilizer of $u$, which contradicts that $u$ has a finite stabilizer group. A similar argument works if there is a point in the rank $0$ locus.

Let $\cV_1 = \cO(-1)|_X$ be the restriction of the tautological subbundle of lines to $X$. Also let $\cV_9 = V_9^*|_X$ and $\cV_5 = \ker \Phi|_X$. Then $\cV_5$ is a rank $5$ vector bundle on $X$ satisfying $\cV_1 \subset \cV_5 \subset \cV_9$.

We now compute the tangent space of $x \in X$. Do a change of basis so that $e_i(x) = 0$ for $i>1$ and so that $(\cV_5)_x$ is defined by $e_i=0$ for $i>5$. Let $R$ be the local ring of $\bP(V_9^*)$ at $x$, and let $\fm$ be its maximal ideal. Over the fiber of $x$, i.e., working modulo $\fm$, the matrix $\Phi$ has rank $4$ and its kernel is the fiber of $\cV_5$, so looks like $\begin{pmatrix} 0 & 0 \\ 0 & \psi \end{pmatrix}$ where $\psi$ is an invertible $4 \times 4$ skew-symmetric matrix. In particular, the determinant of the corresponding $4 \times 4$ block over $R$ does not belong to $\fm$, so is a unit, and hence that block is invertible. So after a change of basis over $R$, we can assume that the matrix over $R$ is of the form $A = \begin{pmatrix} \Phi' & 0 \\ 0 & \Psi \end{pmatrix}$, where $\Phi'$ is $5 \times 5$ and all of its entries belong to $\fm$, and $\Psi = \begin{pmatrix} 0 & -1 & 0 & 0 \\ 1 & 0 & 0 & 0 \\ 0 & 0 & 0 & -1 \\ 0 & 0 & 1 & 0 \end{pmatrix}$.

The $6 \times 6$ Pfaffians of $A$ are the equations that locally cut out $X$ at $x$, and their partial derivatives are the $4 \times 4$ Pfaffians. So any $6 \times 6$ Pfaffian that uses at least $3$ rows from $\Phi'$ has identically $0$ partial derivatives. Hence the only $6 \times 6$ Pfaffians that have nonzero derivatives are those that use $2$ rows from $\Phi'$ together with the last $4$ rows of $A$. The partial derivatives of these Pfaffians are the entries of $\Phi'$, so the tangent space of $x \in X$ is the kernel of the Jacobian map $\cV_9/\cV_1\to \bigwedge^2(\cV_5/\cV_1)$ restricted to $x$. 

Hence it suffices to prove that the Jacobian map is surjective at all points of $X$. Suppose there is a point $x \in X$ so that the map is not surjective. Choose a nonzero linear functional $\lambda$ that annihilates the image. The calculation is equivariant under $\SL((\cV_5/\cV_1)_x)$, so we only need to check what happens for a single representative in each orbit in $\bigwedge^2(\cV_5/\cV_1)^*$. 

If $\lambda$ has rank $2$ (say $\lambda(m)$ is the coefficient of $e_2 \wedge e_3$), it induces a subspace $V_3$ of $(\cV_5)_x$ containing $(\cV_1)_x$, and the $1$-parameter subgroup with weight $(-2,-2,-2,1,1,1,1,1,1)$ is destabilizing. Indeed, before imposing the condition that $\lambda$ annihilates the map to $\bigwedge^2(V_5/V_1)$, the only monomials preventing that weight from destabilizing are $[23i]$ for $4\le i\le 9$; let $\alpha_i$ be the coefficient of $[23i]$. So the $e_2 \wedge e_3$ entry, which is $0$, is $\sum_{i=4}^9 \pm \alpha_i e_i$, so $\alpha_i = 0$ for all $i$.

If $\lambda$ has rank $4$, we similarly find that the weight $(-4,-1,-1,-1,-1,2,2,2,2)$ is destabilizing. Here the only monomials preventing that weight from destabilizing are $[ijk]$ where $\{i,j,k\} \subset \{2,3,4,5\}$, and it is easy to see that they cannot appear. 
\end{proof}

\begin{lemma}
If $u$ is not stable, then $X$ is singular.
\end{lemma}

\begin{proof}
If $u$ is not stable, let $U \subset V_9$ be the corresponding destabilizing subspace as in Proposition~\ref{prop:non-stable-hyper}, i.e., $\dim U = 6$ and $u \in \bigwedge^3 U + \bigwedge^2 U \otimes (V_9/U)$.  Then $U$ cuts out a $\bP^2 \subset \bP(V_9^*)$, and the restriction of $\Phi$ to that plane is supported on $U$.  Thus the intersection of $X$ and that plane is cut out by a single $6 \times 6$ Pfaffian. So either $X$ contains the plane, or it meets it in a cubic curve.

Let $p$ be a point of the intersection.  If $\rank (\Phi|_p) \le 2$, then the $6 \times 6$ Pfaffians of any matrix $\Phi|_p+\eps \Psi$ (here $\eps^2=0$) all vanish, and thus the tangent space of $X$ at $p$ is $8$-dimensional.  If $\rank (\Phi|_p) = 4$, then the tangent space consists of $v$ in $V_9^*/ \langle p \rangle$ such that the restriction of $\Phi(v)$ to the kernel of $\Phi|_p$ is $0$.  Since this restriction is contained in the $5$-dimensional space $\bigwedge^2(U/ \im \Phi|_p) \oplus (U/\im\Phi_p)\otimes (\ker(p)/U)$, it follows that the tangent space is at least $3$-dimensional. Either way, $X$ cannot be a smooth surface. 
\end{proof}

\begin{lemma} \label{lem:coble}
$Y$ is a cubic hypersurface whose singular locus is $X$.
\end{lemma}

\begin{proof}
The fact that $Y$ is a cubic hypersurface follows from \cite[\S 5]{GSW}. We remark that while that paper works over the complex numbers, the particular calculation that $Y$ is a cubic hypersurface is independent of the field since it only relies on knowing that the determinant of the tautological quotient bundle on projective space is the line bundle $\cO(1)$.

It follows from the chain rule that all partial derivatives of the cubic defining $Y$ vanish on $X$, so we just need to show that $Y$ is smooth away from $X$. Let $\cV_1$ denote the restriction of the tautological subbundle of lines to $Y \setminus X$. Note that $\cV_1 = \cO(-1)|_{Y \setminus X}$. Also let $\cV_9 = V_9^*|_{Y \setminus X}$ and $\cV_3 = \ker \Phi|_{Y \setminus X}$. Then $\cV_3$ is a rank $3$ vector bundle on $Y \setminus X$ satisfying $\cV_1 \subset \cV_3 \subset \cV_9$.

The tangent space at a point $x \in Y \setminus X$ is the kernel of the Jacobian map $\cV_9 / \cV_1 \to \bigwedge^2(\cV_3 / \cV_1)$ restricted to $x$ (this is similar to the argument in the previous proof). So $x$ is smooth if and only if this map is nonzero. Suppose that the map is zero at $x$ and do a change of basis so that $e_i(x) = 0$ for $i>1$ and so that $(\cV_3)_x$ is defined by $e_i=0$ for $i>3$. The entries of the Jacobian matrix are given by the coefficients of $[23i]$ for $i=4,\dots,9$, and so those coefficients are $0$. This means that the $1$-parameter subgroup with weight $(-2,-2,-2, 1,1,1, 1,1,1)$ destabilizes $u$, which contradicts that $u$ is stable. So $Y$ is indeed a smooth hypersurface away from $X$.
\end{proof}

Recall that given a variety $X$, its Albanese variety is an abelian variety satisfying a certain universal property (which will not relevant for our purposes).

\begin{proposition} \label{prop:torsor}
  $X$ is a torsor over its Albanese variety $J$ and $\cO_X(1)$ is a $(3,3)$-polarization, i.e., becomes a $(3,3)$-polarization upon passing to the algebraic closure of $\bk$.

  Furthermore, $J$ is indecomposable as a polarized variety, i.e., is not a product of two elliptic curves upon passing to the algebraic closure of $\bk$.
\end{proposition}

\begin{proof}
Let $R$ be a DVR whose residue field is $\bk$ and whose fraction field $K$ is of characteristic $0$. Pick a lift $u_R$ of $u$ to $\bigwedge^3(R^9)$; then $u_K$ is a stable element of $\bigwedge^3(K^9)$ since being non-stable is a closed condition. The construction that we just discussed gives a surface $\cX_R$ over $R$ whose generic fiber $\cX_K$ is a torsor over its Albanese variety \cite[Theorem 5.5]{GSW} and whose special fiber is $\cX_\bk = X$. Let $\ell$ be a prime different from the characteristic of $\bk$. Then the $\ell$-adic Betti numbers of $\cX_K$ and $X$ are the same \cite[Corollary VI.4.2]{milne}. We also know that $\omega_X = \cO_X$ (from the locally free resolution of $\cO_X$ in \cite[\S 5.2]{GSW}). So over $\ks$, $X$ is isomorphic to an abelian surface \cite{bombieri-mumford}. In particular, $X$ is a torsor over its Albanese variety (see the proof of \cite[Theorem 3.1]{GSW}).

The statement about $\cO_X(1)$ is proven in \cite[Proposition 5.6]{GSW} for a field of characteristic $0$. In particular, after base changing to a finite extension of $R$, we can find a cube root of $\cO_X(1)$ over the generic fiber. This can be extended to a line bundle over the whole family whose cube is $\cO_\cX(1)$ (using properness of the Picard variety), which means that it is a $(3,3)$-polarization over the special fiber as well.

For the last statement, note that if $J$ is isomorphic to a product of elliptic curves $E,E'$ as a polarized variety (after passing to the algebraic closure of $\bk$), then the embedding of $X$ into $\bP^8$ is the Segre embedding of the product of $E$ and $E'$ in their plane embeddings. But $X$ the singular locus of a cubic hypersurface, and hence can be set-theoretically cut out by its partial derivatives (quadrics) together with the equation of the cubic. The Segre embedding of two plane cubics requires two cubic equations to be cut out set-theoretically, so they cannot be the same.
\end{proof}

Since $\cO_X(1)$ is a $(3,3)$-polarization, the action of $J[3]$ on $X$ extends to an action of $J[3]$ on $\bP(V_9^*)$. Let $X^i$ be the Picard variety of line bundles on $X$ whose polarization is of type $(i,i)$. By \cite[Theorem 3.6]{GS} (although it is stated in characteristic $0$, the proof does not rely on this assumption, except for the reference to \cite[Proposition 5.6]{GSW}, but see the last paragraph of the previous proof to work around this), we have an isomorphism 
\begin{align*}
X(\ks) &\to X^1(\ks)\\
x &\mapsto \bP(\ker \Phi(x)) \cap X(\ks).
\end{align*}
Since $\Phi$ is defined over $\bk$, this map descends to an isomorphism $X \to X^1$ defined over $\bk$. Furthermore, we have a cubing map $X^1 \to X^3$ and $\cO_X(1) \in X^3$ gives us an isomorphism $X^3 \cong J$. Combining this, we have a map
\[
\psi \colon X \to J
\]
which gives $X$ the structure of a $3$-covering of $J$. 

The preimage of $\cO_X(1)$ under the cubing map $X^1 \to X^3$ is a torsor for $J[3]$. Each geometric point represents a line bundle $\cL$ such that $\rh^0(X; \cL) = 1$, and the zero locus $Z(\cL)$ of the unique, up to scalar multiple, section is a theta divisor of $X$. So $Z(\cL)$ is a genus $2$ curve whose Jacobian is $X$. 

\begin{lemma} \label{lem:ZL}
Under the isomorphism $X \to X^1$, the image of $Z(\cL)$ contains the point representing $\cL$. Furthermore, this point is a Weierstrass point of $Z(\cL)$.
\end{lemma}

\begin{proof}
The first statement is equivalent to $x \in \ker \Phi(x)$. But this follows from the fact that $\Phi(x)$ is the contraction of an alternating trilinear form on $V_9$ by $x$. 

For the second statement, let $P$ be the point on $Z(\cL)$. First assume that the characteristic of $\bk$ is $0$. Then we can check more generally that for any point $x \in X$, we have that $x$ is a Weierstrass point of $\bP(\ker \Phi(x)) \cap X$. For this, it suffices to check a single point since the property is invariant under translation, and this is done in \cite[Remark 3.15]{GS}. 

For the general case, pick a DVR $R$ as in the proof of Proposition~\ref{prop:torsor} and a lift $u_R$ of $u$ to $\bigwedge^3(R^9)$. Our construction is valid in families, so we get a curve $\cC$ over $R$ together with a section $\cP \colon \Spec(R) \to \cC$. Since $\cM = \cO_\cC(\cP)^{\otimes 2}$ extends the canonical bundle on $\cC_K$, we see that $\cM = \Omega^1_{\cC/R}$. In particular, $\cM_\bk = \omega_{Z(\cL)}$, and so $P$ is a Weierstrass point.
\end{proof}

For any two choices $\cL, \cL'$, $Z(\cL)$ and $Z(\cL')$ differ by translation by an element of $J[3]$, so they have the same image under $\psi$. So the reduced image of the union of these curves under $\psi$ is a genus $2$ curve $C \subset J$ (defined over $\bk$) whose Jacobian is $J$ and $P := \cO_X(1) \in C(\bk)$ is a Weierstrass point. 

Using basic properties of finite Heisenberg group schemes, we know that the inclusion $J[3] \subset \PGL(V_9)$ coming from the translation action of $J[3]$ on $\bP(V_9^*)$ lifts to an inclusion $J[3] \subset \SL(V_9) / \bmu_3$.

\begin{lemma} \label{lem:galois}
The kernel of the map of pointed sets $\rH^1(\bk; \SL(V_9)/\bmu_3) \to \rH^1(\bk; \PGL(V_9))$ is trivial, i.e., nontrivial cohomology classes map to nontrivial cohomology classes.
\end{lemma}

\begin{proof}
We have the following commutative diagram
\[
\xymatrix{ 1 \ar[r] & \bmu_3 \ar[r] \ar[d] & \SL_9 \ar[r] \ar[d] & \SL_9 / \bmu_3 \ar[r] \ar[d] & 1 \\
1 \ar[r] & \bG_m \ar[r] & \GL_9 \ar[r] & \PGL_9 \ar[r] & 1 }
\]
which gives the following commutative diagram
\[
\xymatrix{
\rH^1(\bk; \SL_9 / \bmu_3) \ar[r] \ar[d] & \rH^2(\bk; \bmu_3) \ar[d] \\
\rH^1(\bk; \PGL_9) \ar[r] & \rH^2(\bk; \bG_m) }
\]
The horizontal maps have trivial kernel since $\rH^1(\bk; \SL_9) = \rH^1(\bk; \GL_9) = 1$ and the right vertical map has trivial kernel since $\bG_m / \bmu_3 \cong \bG_m$ and $\rH^1(\bk; \bG_m) = 1$. So we conclude that the map $\rH^1(\bk; \SL_9/\bmu_3) \to \rH^1(\bk; \PGL_9)$ has trivial kernel.
\end{proof}

Recall that $3$-coverings $\psi \colon X \to J$ are classified by cohomology classes $[\psi] \in \rH^1(\bk; J[3])$. To get the cohomology class, note that $\psi^{-1}(0)$ is a torsor under $J[3]$.

\begin{lemma} \label{lem:kereta}
$[\psi] \in \ker(\rH^1(\bk; J[3]) \to \rH^1(\bk; \SL(V_9)/\bmu_3))$.
\end{lemma}

\begin{proof}
By Lemma~\ref{lem:galois}, it suffices to show that $[\psi]$ is in the kernel of the composition $\rH^1(\bk; \rJ(C)[3]) \to \rH^1(\bk; \PGL(V_9))$. The map sends the $J[3]$-torsor $\psi^{-1}(0)$ to the $\PGL(V_9)$-torsor $\psi^{-1}(0) \times^{J[3]} \PGL(V_9)$. The data of this $\PGL(V_9)$-torsor is equivalent to the embedding $\psi^{-1}(0) \subset \bP(V_9^*)$. Projective space represents the trivial $\PGL$-torsor, so the image of $[\psi]$ in $\rH^1(\bk; \PGL(V_9))$ is trivial. 
\end{proof}

The trivectors $\gamma_c$ described in Proposition~\ref{prop:S-equiv} are particularly nice for
this construction.  Recall from Proposition~\ref{prop:stable-smooth} that $\gamma_c$ is stable
whenever the corresponding curve $C_c$ is smooth.

\begin{proposition} \label{prop:c-curve}
  Suppose that $C_c$ is smooth.  Then the pair $(C,P)$ corresponding to
  $\gamma_c$ is isomorphic to $(C_c,P_c)$, and the corresponding torsor
  $\psi_c^{-1}(0)$ is trivial.
\end{proposition}

\begin{proof}
  Let $C'$ be the image of $C_c$ in $\bP^8$ under the embedding
  \[
f \colon (x,z) \mapsto [0:0:-1:0:z:0:-z^2:x:z^3].
  \]
  The point $P':=[0:0:0:0:0:0:0:0:1]$ is a Weierstrass point of the closure
  of $C'$, and $C'$ is contained in $\bP(\ker(\Phi(P')))$, so the first
  claim will follow if we can show that $C'$ is contained in the rank 4
  locus. (Indeed, then $C'$ is contained in a theta divisor of
  $X(\gamma_c)$, so must be that theta divisor.)  We may verify that the
  subspace with basis
  \[
  \begin{pmatrix}
    1& 0& 0& z^2& x& -c_{12}z - c_{18}& 0& -c_9 x - c_{24}& 0\\
    0& 1& c_3& -z& 0& -z^2 - c_6 z& -x - c_{15}& -c_3x& 0\\
    0& 0& 1& 0& -z& 0& z^2& -x& 0\\
    0& 0& 0& 0& 0& 1& 0& -z& 0\\
    0& 0& 0& 0& 0& 0& 0& 0& 1
  \end{pmatrix}
  \]
  is in the kernel of $\Phi$ restricted to the point $f(x,z)$, and thus
  that $\Phi|_{C'}$ has rank at most $4$ as required.

  To see that $\psi_c(P')=0$, we need to show that $3\overline{C'}$ is a  section of $\cO_{\bP^8}(1)$.  Since $\overline{C'}$ induces a principal  polarization, the restriction map $\Pic^0(X(\gamma_c))\to \Pic^0(\overline{C'})$ is an isomorphism, and thus it suffices to show that $\cO_{\bP^8}(1)$ and $3\overline{C'}$ have the same restriction to $\overline{C'}$.  In fact, both restrictions are isomorphic to $\cL_{\overline{C'}}(3K_{\overline{C'}})$: the first because $\overline{C'}$ is tricanonically embedded in $\bP^4$, and the second by adjunction and the fact that $K_{X(\gamma_c)}=0$.
\end{proof}

\section{A construction of trivectors} \label{sec:trivector}

Let $C$ be a smooth genus $2$ curve with a marked Weierstrass point $P \in C(\bk)$. Let $\rJ^1(C)$ be the Picard variety of degree $1$ line bundles, and let $\rJ(C)$ be the Jacobian of degree $0$ line bundles. We identify $\rJ^1(C) \cong \rJ(C)$ via $\cL \mapsto \cL(-P)$. 

Define $V_9 = \rH^0(\rJ^1(C); 3\Theta)$. Then $\rJ(C) \subset \bP(V_9^*)$ is embedded by a $(3,3)$-polarization, denoted $\cO(1)$. Define a codimension $1$ subvariety (Poincar\'e divisor) of $\rJ(C) \times \rJ(C)$ by
\[
X = X_{C,P} = \{(\cL_1, \cL_2) \mid \hom_C(\cL_1, \cL_2(P)) \ne 0\}.
\]
The line bundle $\cO(1,1) \otimes \cO(-X)$ has divisor class $3\pi_1^* \Theta + 3\pi_2^* \Theta - \Theta_{\rm diag}$. This is the pullback of a principal polarization on $\rJ(C) \times \rJ(C)$ via the endomorphism 
\begin{align*}
\rJ(C) \times \rJ(C) &\to \rJ(C) \times \rJ(C)\\
(a,b) &\mapsto (2a+b, a+2b).
\end{align*}
The kernel of this map is the diagonal copy of $\rJ(C)[3]$ which has degree $81$. In particular, $\cO(1,1) \otimes \cO(-X)$ has a single cohomology group of dimension $9 = \sqrt{81}$. 

\begin{lemma} \label{lem:bilinear-eqn}
$\rh^0(\cO(1,1) \otimes \cO(-X)) = 9$ and all other cohomology groups vanish.
\end{lemma}

\begin{proof}
It suffices to show that $\rh^0(\cO(1,1) \otimes \cO(-X)) \ne 0$. Define a divisor of $\rJ(C) \times \rJ(C)$:
\begin{align*}
D &= \{(\cL_1, \cL_2) \mid \rh^0(\cL_1 \otimes \cL_2(-P)) \ne 0 \text{ or } \rh^0(\cL_1^{-1} \otimes \cL_2(P)) \ne 0\}.
\end{align*}
Then $D$ is linearly equivalent to $2\pi_1^* \Theta \otimes 2\pi_2^* \Theta$. In particular, $\cO(1,1) \otimes \cO(-D)$ has a nonzero section. But $X \subset D$, so we see that $\cO(1,1) \otimes \cO(-X)$ also has a nonzero section.
\end{proof}

Define 
\[
W = \rH^0(\rJ(C) \times \rJ(C); \cO(1,1) \otimes \cO(-X)) \subset V_9 \times V_9.
\]
By Serre duality and Riemann--Roch, $\hom_C(\cL_1, \cL_2(P)) \ne 0$ if and only if $\hom_C(\cL_2, \cL_1(P)) \ne 0$, so $X$ is preserved under the involution that swaps the two copies of $V_9$.

Let $H$ denote the finite Heisenberg group scheme, i.e., the extension
\[
1 \to \bmu_3 \to H \to \rJ(C)[3] \to 1.
\]
Then $H$ acts diagonally on $V_9 \otimes V_9$ preserving $W$. Note that $V_9$ is the unique irreducible representation of $H$ of weight $1$ (see \cite[Appendix]{sekiguchi}), and $W$ has weight $2$, so $V_9$ and $W^*$ are isomorphic as representations of $H$. So the inclusion gives an $H$-equivariant map (well-defined up to scalar multiple) $V_9^* \to V_9 \otimes V_9$. 

\begin{lemma}
The image of $V_9^*$ is contained in $\bigwedge^2 V_9$.
\end{lemma}

\begin{proof}
By irreducibility, it suffices to show that a single nonzero element in $W$ is alternating under the involution swapping the two copies of $V_9$. Define $D$ as in the proof of Lemma~\ref{lem:bilinear-eqn}. Pick a bilinear equation that vanishes on $D$, i.e., a section of $\cO(1,1) \otimes \cO(-D)$. Since the diagonal $\rJ(C)$ is contained in $D$, if we restrict this equation to the diagonal, we get a section of $4\Theta$ that vanishes on $\rJ(C)$. But we know that such equations are alternating since the Kummer variety has no quadratic polynomials vanishing on it in its $2\Theta$ embedding.
\end{proof}

So we can represent this map by an element $\gamma = \gamma_{(C,P)} \in V_9 \otimes \bigwedge^2 V_9$. 

\begin{lemma}
$\gamma_{(C,P)} \in \bigwedge^3 V_9$.
\end{lemma}

\begin{proof}
Note that $\gamma$ is an $H$-invariant element. Furthermore, $\bigwedge^2 V_9$ is a weight $2$ representation of dimension $36$, and hence it is a direct sum of $4$ copies of $V_9^*$ \cite[Theorem A.6]{sekiguchi}, so the space of $H$-invariant vectors in $V_9 \otimes \bigwedge^2 V_9$ is $4$-dimensional. The space of $H$-invariant vectors in $\bigwedge^3 V_9$ is also $4$-dimensional (we can do this calculation in characteristic $0$ and then specialize to get $\ge 4$-dimensional), so $\gamma \in \bigwedge^3 V_9$.
\end{proof}

\begin{lemma} \label{lem:gamma-rank}
The projection of $X_{C,P}$ to either copy of $\bP^8$ lies in the rank $4$ locus $X(\gamma)$ constructed in \S\ref{sec:3covers}.
\end{lemma}

\begin{proof}
Pick a point $x$ in the projection of $X_{C,P}$ to $\bP(V_9^*)$. Evaluating $\gamma$ on $x$, we get a skew-symmetric matrix $V_9^* \to V_9$ whose image is the set of linear equations vanishing on the fiber of $X_{C,P}$ over $x$. This fiber is a translate of a theta divisor. As an embedded variety, the theta divisor is a genus 2 curve under its tricanonical embedding, and hence satisfies $4$ linear equations, so this skew-symmetric map has rank $4$.
\end{proof}

\begin{proposition} \label{prop:stable-element}
Let $G = \SL(V_9) / \bmu_3$.

\begin{enumerate}[\indent \rm (a)]
\item $\gamma_{(C,P)} \in \bigwedge^3 V_9$ is stable with respect to the action of $G$.

\item The stabilizer of $[\gamma_{(C,P)}] \in \bP(\bigwedge^3 V_9)$ in $G$ is isomorphic to $\rJ(C)[3] \rtimes \Aut(C,P)$ where $\Aut(C,P)$ is the group scheme of $\ks$-automorphisms of $C$ which fix $P$. 

\item If the characteristic is different from $2$ and $5$, then the stabilizer of $\gamma_{(C,P)}$ in $G$ is isomorphic to $\rJ(C)[3]$. 

\item In characteristic $2$, the stabilizer of $\gamma_{(C,P)}$ is isomorphic to $\rJ(C)[3] \rtimes \bZ/2$ if $(C,P)$ is generic, where the $\bZ/2$ comes from the hyperelliptic involution on $C$ and acts by the automorphism $g \mapsto g^{-1}$.

\item In characteristic $5$, the stabilizer of $\gamma_{(C,P)}$ is isomorphic to $\rJ(C)[3]$ if $(C,P)$ is generic.
\end{enumerate}
\end{proposition}

\begin{subeqns}
\begin{proof}
Recall the notation from Proposition~\ref{prop:S-equiv}. By
Proposition~\ref{prop:stable-smooth}, taking any smooth curve $(C_c,P_c)$
guarantees that $\gamma_c$ is stable. The element $\gamma_c$ induces a
system of $9$ bilinear equations in the above way, and we may consider the
resulting (symmetric) subscheme of $X(\gamma_c) \times X(\gamma_c)$.  By
the proof of Lemma~\ref{lem:ZL}, the fiber over any point $x\in
X(\gamma_c)$ is a theta divisor $C_x$ on which $x$ is a Weierstrass
point. Since the associated torsor of $\gamma_c$ is trivial
(Proposition~\ref{prop:c-curve}), this becomes the Poincar\'e divisor on
$\rJ(C_x)\times \rJ(C_x)$ associated to $x$, and it follows that $\gamma_c$
is (projectively) equivalent to $\gamma_{(C_x,x)}$. Furthermore, by
Proposition~\ref{prop:c-curve}, the curve $(C_x,x)$ is isomorphic to
$(C_c,P_c)$. Over $\ks$, every smooth pair $(C,P)$ is isomorphic to
$(C_c,P_c)$ for some $c$, and stability is insensitive to enlarging the
field, so we conclude that $\gamma_{(C,P)}$ is always stable when $(C,P)$
is smooth, which proves (a).

Now we handle (b). First we calculate the stabilizer $G_{[\gamma]}$ of $[\gamma] \in \bP(\bigwedge^3 V_9)$. By functoriality, it is clear that $\rJ(C)[3] \rtimes \Aut(C,P) \subseteq G_{[\gamma]}$. Conversely, let $S$ be a $\bk$-scheme and consider an element $g \in G_{[\gamma]}(S)$. Since $\gamma_{(C,P)}$ is stable, $\rJ(C) \subset \bP(V_9^*)$ is the rank $4$ locus of $\Phi$ and hence is preserved by $g$ (this is true for $\gamma_c$ and $\gamma_{(C,P)}$ is equivalent to $\gamma_c$ over $\ks$). So $g$ preserves the embedding of $(\rJ(C) \times \rJ(C))(S)$ in $(\bP(V_9^*) \times \bP(V_9^*))(S)$ and the subvariety $X(S)$. In particular, $g$ acts on $\rJ(C)(S)$ and preserves the relation $\hom_{C(S)}(\cL_1, \cL_2(P)) \ne 0$, which implies that $g$ permutes the elements of $\rJ(C)[3](S)$. So $G_{[\gamma]}$ is generated by $\rJ(C)[3]$ and a subgroup of the automorphisms of $\rJ(C)$ which fixes the identity. Using Torelli's theorem, an automorphism of $\rJ(C)$ that fixes the identity and the embedding of $\rJ(C)$ comes from an automorphism of $C$ which fixes $P$; since $G_{[\gamma]}$ contains $\rJ(C)[3] \rtimes \Aut(C,P)$, we deduce that they are equal. This proves (b).

In particular, $G_{[\gamma]}$ is finite. Let $\lambda \colon G_{[\gamma]} \to \bG_m$ be the eigenvalue associated with the action of $G_{[\gamma]}$ on $[\gamma]$. The stabilizer of $\gamma$ is $\ker \lambda$. First note that $\rJ(C)[3] \subseteq \ker \lambda$ since the projective action of $\rJ(C)[3]$ lifts to a linear action of the Heisenberg group scheme $H$ in $\SL(V_9)$, and we have already explained why $H$ acts trivially on $\gamma$.

Now we prove (c), so we assume that the characteristic is different from $2$ and $5$. Recall that $G_\gamma = \ker \lambda$, so we need to show that $\Aut(C,P)$ is mapped faithfully via $\lambda$. Put $(C,P)$ into Weierstrass normal form
\begin{align} \label{eqn:proof-weier}
y^2 = x^5+c_{12}x^3+c_{18}x^2+c_{24}x+c_{30}.
\end{align}
By degree considerations (where $\deg(y)=5$ and $\deg(x)=2$), any automorphism of $(C,P)$ must be of the form
\begin{align*}
y&\mapsto a_1^5 y + a_2 x^2+a_3 x + a_4\\
x&\mapsto a_1^2 x + a_5
\end{align*}
for some scalars $a_1, a_2, a_3, a_4, a_5$ and $a_1 \ne 0$. When we do these substitutions to \eqref{eqn:proof-weier} and subtract \eqref{eqn:proof-weier}, we get a relation on $x,y$ which is of degree $<10$, so which must be identically $0$. The coefficients of $x^2y, xy, y$ on the left side are $2a_1^5 a_2$, $2a_1^5 a_3$, $2a_1^5 a_4$, respectively, so we conclude that $a_2 = a_3 = a_4 = 0$. Similarly, the coefficient of $x^4$ on the right side is $5a_1^8 a_5$, so we conclude that $a_5 = 0$. 

In particular, the automorphism takes the form
\[
y \mapsto a_1^5 y, \qquad x \mapsto a_1^2 x
\]
for some $\ell$th root of unity $a_1$ (since the automorphism has finite order). Again, do the substitution to \eqref{eqn:proof-weier}, divide by $a_1^{10}$ and subtract \eqref{eqn:proof-weier}. Then we get 
\[
c_{12} (a_1^{-4} - 1) x^3 + c_{18} (a_1^{-6} - 1) x^2 + c_{24}(a_1^{-8} - 1) x + c_{30}(a_1^{-10} - 1) = 0, 
\]
so the left hand side must be identically $0$. If $\ell \notin \{1,2,4,5,8,10\}$, then $c_{12} = c_{24} = c_{30} = 0$. But then \eqref{eqn:proof-weier} is $y^2 = x^2 (x^3 + c_{18})$, which is a singular curve. So we only need to show that $\lambda$ maps $\bmu_\ell \subset \Aut(C,P)$ faithfully where $\ell \in \{1,2,4,5,8,10\}$; it suffices to consider the cases $\ell = 2$ and $\ell = 5$.

For $\ell\in \{2,5\}$, let $\sM_\ell$ be the space of curves with an action
of $\bmu_\ell$ as described above. Then $\sM_\ell$ is an irreducible stack
over $\bZ[1/\ell]$.  Indeed, the action of $\bmu_\ell$ must survive
completing the $\ell$-th power in the curve, and this forces the action to
be diagonal in the variables.  Thus $\sM_5$ is the irreducible stack of
curves of the form $x^2+z^5+c_{15}x+c_{30}=0$ modulo $x\mapsto x+a$ (with
$\bmu_5$ acting by $z\mapsto \zeta_5 z$)
and $\sM_2$ is the irreducible stack of curves of the form
\[
x^2+z^5+c_6 z^4+c_{12}z^3+c_{18} z^2+c_{24}z+c_{30},
\]
modulo $z\mapsto z+a$ (with $\bmu_2$ acting by $x\mapsto -x$).

Let $\sC$ be the universal curve over $\sM_\ell$. By composing $\lambda$ with the natural morphism $\bmu_\ell\to \Aut(\sC)$, we obtain a scheme morphism from $\sM_\ell$ to the dual group $\bmu_\ell^\vee \cong \bZ/\ell$. Since $\sM_\ell$ is irreducible, this morphism must be constant, and thus may be computed in characteristic $0$. In this case, any point in the Cartan subspace which is not in the union of the reflection hyperplanes has a trivial stabilizer. In particular, the stabilizer of $\gamma$ is isomorphic to $\rJ(C)[3]$. So faithfulness of $\lambda$ in characteristic $0$ implies faithfulness of the restriction to $\bmu_\ell$ over $\bZ[1/\ell]$. 

If $(C,P)$ is generic, then $\Aut(C,P) \cong \bZ/2$ and is generated by the hyperelliptic involution $\iota_C$ (via Torelli's theorem, this is equivalent to the statement that the generic principally polarized Jacobian has automorphism group $\bZ/2$, which is \cite[Lemma 11.2.6]{katz-sarnak}). The induced action of $\iota_C$ on $\rJ(C)[3]$ is the inverse map and, if $\bk$ has characteristic $0$, we can calculate explicitly in a standard Cartan (see, for example, \cite[(3.2)]{GS}) that $\lambda(\iota_C) = -1$, so $\iota_C \notin \ker \lambda$. By semicontinuity, the same is true in any characteristic different from $2$. In characteristic $2$, the restriction of $\lambda$ to $\iota_C$ is trivial since its image is in $\bmu_2 \subset \bG_m$, which is non-reduced, while $\iota_C$ generates a subgroup isomorphic to $\bZ/2$. So $\iota_C \in \ker \lambda$. This proves (d) and (e).
\end{proof}
\end{subeqns}

\section{Putting it all together} \label{sec:bijection}

Let $G = \SL(V_9) / \bmu_3$ and let $G_\gamma$ be the stabilizer subgroup of $\gamma \in \bigwedge^3 V_9$.

\begin{proposition} \label{prop:same-curve}
Pick $(C,P)$ and $(C',P')$ so that we have elements $\gamma = \gamma_{(C,P)} \in \bigwedge^3 V_9$ and $\gamma' = \gamma_{(C',P')} \in \bigwedge^3 V_9'$. Suppose that there is a linear isomorphism $\phi \colon V_9 \cong V_9'$  that sends the line generated by $\gamma_{(C,P)}$ to the line generated by $\gamma_{(C',P')}$. Then there exists an isomorphism $(C,P) \cong (C',P')$.
\end{proposition}

\begin{proof}
Using $\phi$, we can embed $X_{C,P}$ and $X_{C',P'}$ in the same $\bP^8\times \bP^8$, in such a way that their images satisfy the same $9$ bilinear equations $W_\gamma = W_{\gamma'}$. Now, consider the projection $\pi$ onto the first $\bP^8$. By Lemma~\ref{lem:gamma-rank}, the image of $X_{C,P}$ in $\bP^8$ maps into the rank $4$ locus $X(\gamma)$, which is a torsor over an abelian surface (Proposition~\ref{prop:torsor}). The fibers of $\pi$ are curves, and so the image of $\pi$ is a surface. Since $X(\gamma)$ is irreducible, the image must be equal to $X(\gamma)$. In particular, $\pi$ gives an identification $\rJ(C) = X(\gamma)$. The same applies to $X_{C',P'}$, so in particular, we find that $\phi$ defines an isomorphism $\rJ(C)\cong \rJ(C')$ which identifies the respective $3\Theta$ line bundles. Finally, we can recover $C$ as $\pi^{-1}(0)$ under $\pi \colon X_{C,P} \to \rJ(C)$, and $P$ as the point $(0,0) \in X_{C,P} \subset \rJ(C) \times \rJ(C)$, and similarly for $(C',P')$.
\end{proof}

\begin{proposition} \label{prop:trivial-torsor}
If we apply the construction of \S\ref{sec:3covers} to $\gamma_{(C,P)}$, then the torsor $X(\gamma_{(C,P)})$ is trivial, and $(C,P)$ is the marked curve that comes from the construction in \S\ref{sec:3covers}.
\end{proposition}

\begin{proof}
This was shown in the proof of Proposition~\ref{prop:same-curve}.
\end{proof}

\begin{lemma} \label{lem:bilinear-recovery}
Let $\gamma \in \bigwedge^3 V_9$ be a stable element. Then $\gamma$ can be recovered from the $9$-dimensional space of bilinear forms $W_{\gamma} \subset \bigwedge^2 V_9$ up to scalar multiple.
\end{lemma}

\begin{proof}
We represent this space as an injective map $f \colon W_\gamma \to \bigwedge^2 V_9$. Since $\gamma$ is stable, the locus $Y(\gamma) = \{x \in \bP(W_\gamma) \mid \rank f(x) \le 6\}$ is a cubic hypersurface (Lemma~\ref{lem:coble}), and so the generic element in $W_\gamma$ has rank $8$, and hence its kernel is a line in $V_9^*$. This gives a rational map $\rho \colon \bP(W_\gamma) \dashrightarrow \bP(V_9^*)$. Furthermore, $\rho$ is the projectivization of a linear map $\phi \colon W_\gamma \to V_9^*$ since there exists an identification $W_\gamma = V_9^*$ coming from $\gamma \colon V_9^* \to \bigwedge^2 V_9$. This linear map is unique up to scalar multiple, and our goal is to reconstruct it from $W_\gamma$.

Pick $10$ elements in $W_\gamma$ with rank $8$ such that any $9$ of them are linearly independent. Pick a basis $e_1,\dots,e_9$ for $W_\gamma$. Up to projective equivalence, we may assume that the points are the projectivizations of $e_1,\dots,e_9, e_1+\cdots+e_9$. For $i=1,\dots,9$, choose $x_i \in \rho(e_i)$ such that $x_1+\cdots+x_9 \in \rho(e_1+\cdots+e_9)$. This can be used to define a linear map $\phi' \colon W_\gamma \to V_9^*$ which is well-defined up to a global choice of scalar. In particular, there must be scalars $\alpha_i$ such that $\phi'(e_i) = \alpha_i \phi(e_i)$ for $i=1,\dots,9$. However, $\phi'(e_1+\cdots+e_9) = \alpha_1 \phi(e_1) + \cdots + \alpha_9 \phi(e_9)$, and it must generate the same line as $\phi(e_1+\cdots+e_9)$, so we conclude that $\alpha_1=\cdots=\alpha_9$ and hence $\phi$ and $\phi'$ agree up to scalar multiple.
\end{proof}

\begin{proposition} \label{prop:same-orbit}
Pick stable elements $\gamma, \gamma' \in \bigwedge^3 V_9$ with trivial cohomology class, i.e., $[\psi] = [\psi'] = 0$. Let $(C,P)$ and $(C',P')$ be the marked curves constructed in \S\ref{sec:3covers} and assume that there is an isomorphism $(C,P) \cong (C',P')$ defined over $\bk$. Then the lines spanned by $\gamma$ and $\gamma'$ are in the same $\PGL(V_9)$-orbit.
\end{proposition}

\begin{proof}
Via the construction in \S\ref{sec:3covers}, we have torsors $X(\gamma), X(\gamma') \subset \bP(V_9^*)$. Since its cohomology class is trivial, we can find a $\bk$-rational point in the preimage of $0$ under the $3$-covering $X(\gamma) \to \rJ(C)$. Use this point, call it $0$, to identify $X(\gamma)$ with $\rJ(C)$. The construction shows that $X(\gamma)$ has a $\bk$-rational theta divisor $C \subset X(\gamma)$ such that $0 \in X(\gamma)$ is a Weierstrass point on $C$. These remarks also apply to $C' \subset X(\gamma')$. 

The embedding $X(\gamma) \subset \bP(V_9^*)$ can be reconstructed from the data of $(C,P)$. In particular, the isomorphism $(C,P) \cong (C',P')$ that is assumed to exist induces an isomorphism $X(\gamma) \cong X(\gamma')$ that preserves their embeddings into $\bP(V_9^*)$. So, up to a linear change of coordinates for one of the embeddings, we have $X(\gamma) = X(\gamma')$. In particular, there is an identification of their Poincar\'e divisors, which then satisfy the same $9$ bilinear equations, i.e., $W_\gamma = W_{\gamma'}$. Lemma~\ref{lem:bilinear-recovery} implies that $\gamma$ and $\gamma'$ are equal up to scalar multiple after the change of coordinates.
\end{proof}

\begin{theorem} \label{thm:genus2-param}
The construction in \S\ref{sec:3covers} is a bijection between the stable orbits of $\bP(\bigwedge^3 V_9)$ under the action of $\PGL(V_9)$ and the set of $\bk$-isomorphism classes of triples $(C,P,\psi)$ where $C$ is a smooth genus $2$ curve, $P \in C(\bk)$ is a Weierstrass point, and $\psi \in \ker(\rH^1(\bk; \rJ(C)[3]) \to \rH^1(\bk; \PGL(V_9)))$.
\end{theorem}

\begin{proof}
Given a smooth genus $2$ curve with Weierstrass point $P$, we have constructed a stable element in $\bP(\bigwedge^3 \rH^0(\rJ(C); 3\Theta)^*)$ in \S\ref{sec:trivector}. If we pick a linear isomorphism $\rH^0(\rJ(C); 3\Theta)^* \cong V_9$, we hence get an element of $\bP(\bigwedge^3 V_9)$. The $\PGL(V_9)$-orbit of this element does not depend on the choice of isomorphism. So we have a well-defined map $\Phi$ from the set of $\bk$-isomorphism classes of $(C,P)$ to $\PGL(V_9)$-orbits in $\bP(\bigwedge^3 V_9)$. Furthermore, by Proposition~\ref{prop:trivial-torsor}, $\Phi(C,P)$ has trivial cohomology class. By Proposition~\ref{prop:same-curve}, this map is injective on $\bk$-isomorphism classes of $(C,P)$.

In \S\ref{sec:3covers}, we constructed a map from $\PGL(V_9)$-orbits of $\bP((\bigwedge^3 V_9)_\st)$ to the set of $\bk$-isomorphism classes of $(C,P)$; let $\Psi$ be the restriction to the orbits with trivial cohomology class. By Proposition~\ref{prop:trivial-torsor}, $\Psi \circ \Phi$ is the identity, so $\Psi$ is surjective. By Proposition~\ref{prop:same-orbit}, $\Psi$ is injective, so $\Phi$ is a bijection between $\bk$-isomorphism classes of marked curves $(C,P)$ and $\PGL(V_9)$-orbits of stable elements in $\bP(\bigwedge^3 V_9)$ with trivial cohomology class.

By Proposition~\ref{prop:stable-element}, the stabilizer of any element in $\Phi(C,P)$ is isomorphic to $\rJ(C)[3] \rtimes \Aut(C,P)$. In particular, 
\[
\ker(\rH^1(\bk; \rJ(C)[3] \rtimes \Aut(C,P)) \to \rH^1(\bk; \PGL(V_9)))
\]
is in bijection with the $\PGL(V_9)$-orbits in $\bP(\bigwedge^3 V_9)$ which are in the same orbit as $\Phi(C,P)$ over a separable closure of $\bk$. Now consider the map 
\[
\rH^1(\bk; \rJ(C)[3] \rtimes \Aut(C,P)) \to \rH^1(\bk; \Aut(C,P)).
\]
The latter group parametrizes $\bk$-forms of $C$, so each such orbit is naturally associated to a $\bk$-form of $C$. In particular, the orbits that correspond to $C$ itself, i.e., $\bk$-forms that are actually isomorphic to $C$ over $\bk$, are in bijection with 
\[
\ker(\rH^1(\bk; \rJ(C)[3]) \to \rH^1(\bk; \PGL(V_9))).
\]
In particular, $\Phi$ extends to a map on triples $(C,P,\psi)$ and gives an isomorphism to all stable $\PGL(V_9)$-orbits in $\bP(\bigwedge^3 V_9)$.
\end{proof}

\begin{corollary}
If $\bk$ is algebraically closed of characteristic different from $3$, then every stable element of $\bigwedge^3 V_9$ is in the standard Cartan subspace up to the action of $G$.
\end{corollary}

\begin{proof}
By the construction in \S\ref{sec:trivector}, every stable element of the form $\gamma_{(C,P)}$ is in the standard Cartan subspace up to the action of $G$. By Theorem~\ref{thm:genus2-param}, they all arise in this way.
\end{proof}

\section{Complements} \label{sec:complements}

\subsection{Selmer groups} \label{sec:selmer}

For this section, suppose that $\bk$ is a global field, and let $B$ be an abelian variety defined over $\bk$. Let $\alpha \in \rH^1(\bk; B[n])$ be a torsor for $B[n]$. We can use this to twist the multiplication by $n$ map $B \xrightarrow{\cdot n} B$ to get $B' \to B$ where $B'$ is a $B$-torsor. We say that $\alpha$ is an element of the \df{$n$-Selmer group} of $B$ if, for all completions $\bk_v$ of $\bk$, the corresponding torsor $B'$ has a $\bk_v$-rational point. We denote this subgroup by $\Sel_n(B) \subset \rH^1(\bk; B[n])$. 

\begin{proposition} \label{prop:selmer}
Let $C$ be a genus $2$ curve with rational Weierstrass point. The $3$-Selmer group $\Sel_3(\rJ(C))$ is contained in $\ker(\rH^1(\bk; \rJ(C)[3]) \to \rH^1(\bk; \SL_9/\bmu_3))$.
\end{proposition}

\begin{proof}
Pick $\psi \in \Sel_3(\rJ(C))$. Then $\psi$ gives an embedding $X \subset S$ where $S$ is a Brauer--Severi variety of dimension $8$ and $X$ is the corresponding twist of $\rJ(C)$. By assumption, $X(\bk_v) \ne \emptyset$ for all completions $\bk_v$ of $\bk$. Brauer--Severi varieties satisfy the Hasse principle, so we conclude that $S \cong \bP^8$ and that the image of $\psi$ in $\rH^1(\bk; \PGL(9))$ is trivial. By Lemma~\ref{lem:galois}, its image in $\rH^1(\bk; \SL(9)/\bmu_3)$ is also trivial.
\end{proof}

In particular, triples $(C,P,\psi)$ where $C$ is a genus $2$ curve, $P \in C(\bk)$ is a Weierstrass point, and $\psi \in \Sel_3(\rJ(C))$ are parametrized by certain $\PGL(V_9)$-orbits in $\bP(\bigwedge^3 V_9)$.

\subsection{Ordinary curves}

Let $\bk$ be a field of characteristic $3$. Given a smooth curve of genus $g$, then $|\rJ(C)(\ks)| = 3^r$ where $0 \le r \le g$. The quantity $r$ is the \df{$3$-rank} of the curve. If $r=g$, then $C$ is \df{ordinary}. 

The Lie algebra of type $\rE_8$ has a cubing map $x \mapsto x^{[3]}$ which induces a cubing map $\bigwedge^3 V_9 \to \fsl(V_9)$.

Set $\gamma_0$ to be the principal nilpotent element with all $c_i=0$ in Proposition~\ref{prop:S-equiv}:
\[
\gamma_0 = [267]+[258]+[348]+[169]+[357]+[249]+[178]+[456].
\]

\begin{lemma} \label{lem:lie-algebra-char3}
Let $C$ be the genus $2$ curve associated with a stable element $\gamma \in \bigwedge^3 V_9$. The Lie algebra of the stabilizer of $\gamma$ $($equivalently, the Lie algebra of $\rJ(C)[3])$ is $2$-dimensional, and is spanned by $\gamma^{[3]}$ and $\gamma^{[9]}$.
\end{lemma}

\begin{proof}
Consider the height grading on $\fe_8$ discussed in \S\ref{ss:cartan}. The principal nilpotent element in $\fe_8$ restricts to $\gamma_0$ and \cite[Theorem 2.6]{springer} shows that, outside of degrees $-1,0$, multiplication by $\gamma_0$ has a single kernel element in characteristic $3$ in degrees $3,9,-4,-10$. Reducing these degrees modulo $3$, we see that $\ker(\ad \gamma_0 \cap \fsl(V_9))$ has two elements coming from degrees $3$ and $9$, which are $\gamma_0^{[3]}$ and $\gamma_0^{[9]}$, together with whatever comes from degree $0$. However, the latter is $0$ since the structure constants of the Lie algebra $\fe_8$ are all $\pm 2$.

By semicontinuity, the Lie algebra of $\rJ(C)[3]$ coming from $\gamma$ is at most $2$-dimensional. However, a generic $\gamma$ (for example, take a stable element in the Cartan subspace) comes from an ordinary curve, in which case the Lie algebra is $2$-dimensional, so the dimension is always $2$, and agrees with the $\ge 2$-dimensional span of $\gamma^{[3]}$ and $\gamma^{[9]}$.
\end{proof}

\begin{corollary}
Pick a stable element $\gamma \in \bigwedge^3 V_9$. The $3$-rank of the associated curve is:
\begin{enumerate}[\indent \rm (a)]
\item $2$ if $\gamma^{[3]}$ is semisimple,
\item $1$ if $\gamma^{[3]}$ is not semisimple, but $\gamma^{[9]}$ is semisimple,
\item $0$ if neither $\gamma^{[3]}$ nor $\gamma^{[9]}$ is semisimple.
\end{enumerate}
\end{corollary}

\begin{proof}
The Weil pairing shows that $\rJ(C)[3]^\vee \cong \rJ(C)[3]$. In particular,  the $3$-rank $r$ appears in the reduced quotient $(\bZ/3)^r$ of $\rJ(C)[3]$ and hence appears in the largest diagonalizable subgroup $\bmu_3^r \subset \rJ(C)[3]$. So the $3$-rank of the curve $C$ is the dimension of the largest semisimple subalgebra of the Lie algebra of $\rJ(C)[3]$. 
\end{proof}

\begin{remark}
We can write our curve $C$ in Weierstrass normal form
\[
y^2 = x^5 + c_{12} x^3 + c_{18} x^2 + c_{24} x + c_{30}.
\]
According to \cite[Lemma 2.2]{elkin-pries}, the $3$-rank of $C$ is:
\[
\begin{cases}
2 & \text{if $c_{24}\ne 0$},\\
1 & \text{if $c_{24}=0,\ c_{18}\ne 0$},\\
0 & \text{if $c_{24}=c_{18}=0$.} 
\end{cases} 
\]
Furthermore, by Lemma~\ref{lem:lie-algebra-char3}, we know that $\gamma^{[27]}$ is a linear combination of $\gamma^{[3]}$ and $\gamma^{[9]}$; in Weierstrass normal form, a computer calculation shows that
\[
\gamma^{[27]} = c_{24} \gamma^{[3]} - c_{18} \gamma^{[9]}.  \qedhere
\]
\end{remark}

\subsection{Model for $3$-torsion}

Let $\gamma \in \bigwedge^3 V_9$ be a stable vector. By Proposition~\ref{prop:stable-element}, the stabilizer of $[\gamma] \in \bP(\bigwedge^3 V_9)$ in $\SL(V_9)/\bmu_3$ is isomorphic to $\rJ(C)[3] \rtimes \Aut(C,P)$ where $(C,P)$ is the marked curve associated to $\gamma$, and there is also an associated torsor of $\rJ(C)[3]$. Here is a more direct construction for this torsor.

The split Lie algebra of type $\rE_8$ has a graded direct sum decomposition
\[
\fsl(V_9) \oplus \bigwedge^3 V_9 \oplus \bigwedge^6 V_9 .
\]
Pick a flag of subspaces $F_1 \subset F_3 \subset F_6 \subset F_8 \subset V_9$ (the subscripts indicate the dimension of the subspace). Via the embedding $\Flag(1,8;V_9) \subset \bP(\fsl(V_9))$, the subspaces $F_1 \subset F_8$ determine (up to scalar multiple) an element $v_0 \in \fsl(V_9)$, $F_3$ determines an element $v_1 \in \bigwedge^3 V_9$, and $F_6$ determines an element $v_2 \in \bigwedge^6 V_9$. We say that $F_\bullet$ is \df{compatible with $\gamma$} if:
\begin{enumerate}[\indent \rm (a)]
\item $[v_0, \gamma] \in \bigwedge^3 V_9$ is a scalar multiple of $v_1$,
\item $[v_1, \gamma] \in \bigwedge^6 V_9$ is a scalar multiple of $v_2$, and
\item $[v_2, \gamma] \in \fsl(V_9)$ contains $F_6$ in its kernel and its image is contained in $F_1$.
\end{enumerate}

The conditions above are algebraic, so determines a subscheme $F(\gamma)$ of compatible flags. To be precise, let $F_\bullet$ be the standard flag defined by $F_i = \langle e_1, \dots, e_i\rangle$. If it is compatible with $\gamma$, then it implies that the coefficient of $e_i \wedge e_j \wedge e_k$ vanishes where $ijk$ is 
\begin{equation} \label{eqn:monomials}
\begin{split}
ij9, \qquad & 4 \le i<j \le 8;\\
ij9, \qquad & i=2,3;\quad 4 \le j \le 8;\\
i78, \qquad & 2 \le i \le 6;\\
ij7, ij8, \qquad & 4 \le i<j \le 6. 
\end{split}
\end{equation}
Let $P$ be the stabilizer in $\GL(V_9)$ of the standard flag $F_1 \subset F_3 \subset F_6 \subset F_8$. The span of the monomials which are not listed above forms a $P$-submodule of $\bigwedge^3 V_9$, and via algebraic induction from $P$ to $\GL(V_9)$, we get a subbundle $\xi \subset \bigwedge^3 V_9 \times \Flag(1,3,6,8; V_9)$. Hence, $\gamma$ is a section of the quotient bundle $\eta$, which is of rank $31$, and $F(\gamma)$ is the zero locus of this section, and this can be used to define it as a scheme. 

Define a $\GL(V_9)$-equivariant map $\pi \colon \Flag(1,3,6,8; V_9) \to \bP(V_9^*)$ by sending $F_\bullet$ to the annihilator of $F_8$ in $V_9^*$.

\begin{lemma} \label{lem:pi-bij}
$\pi(F(\gamma))$ is the underlying set of the torsor for $\rJ(C)[3]$ constructed previously. In fact, $\pi$ gives a bijection between the underlying sets.
\end{lemma}

\begin{proof}
Fix a compatible flag $F_\bullet$. Pick nonzero $u \in F_1$ and pick nonzero $x \in V_9^*$ which annihilates $F_8$. The action of $v_0$ on $\gamma$ can be obtained by first contracting $\gamma$ by $x$ and then multiplying by $u$. By assumption, the result is a pure trivector, say equal to $u \wedge \ell_1 \wedge \ell_2$ for $\ell_1, \ell_2 \in V_9$. Let $\Phi(x)$ be the contraction of $\gamma$ by $x$. Then $\Phi(x) = \ell_1 \wedge \ell_2 + \ell_3 \wedge u$ for some $\ell_3 \in V_9$ and so $x \in X(\gamma)$. So $\ker \Phi(x) \cap X$ gives a divisor $D$ on $X$ by \cite[Theorem 3.6]{GS}, and we want to show that $3D$ is the divisor corresponding to $\cO_X(1)$.

To do this, it suffices to show that there is a hyperplane $H \subset \bP(V_9^*)$ such that $H \cap X = \ker \Phi(x) \cap X$ as sets. We claim that this works if $H$ is the zero locus of $u$. It follows from the definition that $\ker \Phi(x) \cap X \subseteq H \cap X$. The correctness of this statement is unaffected if we do a change of basis and if we pass to an algebraic closure of $\bk$. So we do both and assume that $F_\bullet$ is the standard flag. This implies that $\Phi(x) = e_1 \wedge \ell + e_2 \wedge e_3$ where $\ell$ is in the span of $e_2, \dots, e_8$ but is not contained in the span of $e_2$ and $e_3$. In particular, doing a further change of basis using the stabilizer of $F_\bullet$, we may assume that $\ell = e_6$ or $\ell = e_8$. In both cases, we can verify, for generic $\gamma$, using a computer algebra system, that if $y \in H \cap X$, then $y \in \ker \Phi(x)$. The general case follows because $H \cap X$ contains $C$ so cannot possibly degenerate any further unless it increases in dimension (but $X$ is not contained in a hyperplane).

For the last statement, let $x$ be a $\ks$-point in the image of $\pi$. From the proof above, we see that $x$ determines the subspace $F_8$ in the flag. Also, there is a hyperplane $H \subset \bP(V_9^*)$ such that $H \cap X = \ker \Phi(x) \cap X$ as sets. Since $X$ is not contained in a hyperplane, $H$ is unique with this property, and it determines $F_1$. If $F_\bullet$ is a compatible flag, then $F_3$ is determined by $F_1 \subset F_8$ and $F_6$ is determined by $F_3$, so we are done.
\end{proof}

\begin{theorem}
$F(\gamma)$ is a degree $81$ scheme of dimension $0$. In particular, $\pi$ restricts to a $\rJ(C)[3]$-equivariant isomorphism between $F(\gamma)$ and the torsor for $\rJ(C)[3]$, so $F(\gamma)$ is reduced outside of characteristic $3$.
\end{theorem}

\begin{proof} 
Note that $\dim \Flag(1,3,6,8; V_9) = 31 = \rank(\eta)$, and Lemma~\ref{lem:pi-bij} shows that $F(\gamma)$ is $0$-dimensional whenever $\gamma$ is stable. Hence the degree of $F(\gamma)$ can be calculated as the top Chern class of $\eta$, which can be shown to be $81$ as follows. The Borel presentation for the (rational) Chow ring of $\Flag(1,3,6,8; V_9)$ describes it as the subring of $S_1 \times S_2 \times S_3 \times S_2 \times S_1$-invariants inside the quotient ring $\bQ[x_1,\dots,x_9] / I$ where $I$ is generated by all positive degree homogeneous $S_9$-invariants ($S_9$ is the symmetric group on $9$ letters, and acts by permuting the $x_i$; $S_1 \times S_2 \times S_3 \times S_2 \times S_1$ is the subgroup where the first $S_2$ permutes $x_2, x_3$, $S_3$ permutes $x_4,x_5,x_6$, and the second $S_2$ permutes $x_7,x_8$). Over the full flag variety of $V_9$, the bundle $\eta$ is filtered by line bundles, one for each monomial in \eqref{eqn:monomials} and the (rational) Chow ring of the full flag variety is $\bQ[x_1, \dots, x_9] / I$. In the Borel presentation, the Chern class of the line bundle corresponding to the monomial $ijk$ is represented by $x_i + x_j + x_k$. So the top Chern class of $\eta$ is the product of these linear forms, which is $81 m$ modulo $I$ where $m$ is a nonzero monomial of degree $31$ (doing this in Macaulay2 \cite{M2}, we get $81 x_2  x_3 x_4^3 x_5^3 x_6^3 x_7^6 x_8^6 x_9^8$).

The $\rJ(C)[3]$-equivariance of $\pi$ comes from the fact that $\rJ(C)[3]$ is a subgroup of the stabilizer of $\gamma$.
\end{proof}

In particular, in every $G(\ks)$-orbit of a point in $\bigwedge^3 V_9$, there is a distinguished $G(\bk)$-orbit corresponding to elements which have a compatible flag defined over $\bk$.

\begin{corollary}
In characteristics different from $2$ and $5$, $\SL(V_9)/\bmu_3$ acts freely on the scheme of pairs $(\gamma, F_\bullet)$ where $\gamma \in \bigwedge^3 V_9$ is stable and $F_\bullet$ is a compatible flag for $\gamma$.
\end{corollary}

\begin{remark}
If we permute the basis via 974852631, then the family in Proposition~\ref{prop:S-equiv} becomes
\begin{align*}
[348]-[357]+[267]-[189]+[456]+[239]-[147]-[258]\\
+c_3 [345]
-c_6 [234]
+c_9 [127]
-c_{12} [124]
-c_{15} [356]
-c_{18} [236]
+c_{24} [126]
-c_{30} [136]
\end{align*}
and the standard coordinate flag is compatible with the entire family.
\end{remark}

\end{document}